\documentclass[11pt,leqno]{amsart}
\usepackage{amsmath,amsthm,amsfonts,amssymb,amscd}
\newtheorem{theorem}{Theorem}

{\newtheorem{Theorem}{Theorem}[section]
\newtheorem{Corollary}[Theorem]{Corollary}
\newtheorem{Proposition}[Theorem]{Proposition}
\newtheorem{Lemma}[Theorem]{Lemma}
\newtheorem{Claim}[Theorem]{Claim}
\theoremstyle{Definition}
\newtheorem{Definition}[Theorem]{Definition}
\newtheorem{Example}[Theorem]{Example}

\theoremstyle{Remark}
\newtheorem{Remark}[Theorem]{Remark}

\def\leaderfill{\leaders\hbox to .8em{\hss .\hss}\hfill}
\def\_#1{{\lower 0.7ex\hbox{}}_{#1}}

\def\C{{\mathcal{C}}}
\def\L{{\mathcal{L}}}
\def\G{{\mathcal{G}}}
\def\fa{{\mathcal{F}}}
\def\O{{\mathcal{O}}}

\def\M{{\mathcal{M}}}

\def\U{{\mathcal{U}}}

\def\po{{\partial}}

\def\Om{{\Omega}}
\def\vr{{\varphi}}
\def\ga{{\gamma}}
\def\Ga{{\Gamma}}
\def\la{{\lambda}}
\def\La{{\Lambda}}
\def\ov{\overline}
\def\al{{\alpha}}

\def\lg{{\langle}}
\def\rg{{\rangle}}

\def\bh{{\Bbb{H}}}

\def\re{{\Bbb{R}}}
\def\bz{{\Bbb{Z}}}

\def\bd{{\Bbb{D}}}
\def\bc{{\Bbb{C}}}
\def\bn{{\Bbb{N}}}
\def\be{{\beta}}

\def\SL{\operatorname{{SL}}}
\def\Res{\operatorname{{Res}}}

\def\Aut{\operatorname{{Aut}}}
\def\SL{\operatorname{{SL}}}
\def\Aff{\operatorname{{Aff}}}
\def\Hol{\operatorname{{Hol}}}

\def\Ker{\operatorname{{Ker}}}

\def\Diff{\operatorname{{Dif}}}
\def\sing{\operatorname{{sing}}}

\title[On  transversely holomorphic foliations]{On transversely holomorphic foliations with
homogeneous transverse structure}

\author{Liliana Jurado}
\author{Bruno Sc\'ardua}
\address{Liliana Jurado \\
IMCA - Inst. Mat. Ciencias Afines - Calle los Bi�logos 245, La Molina, Lima - Per\'u}
\email{lilianajurado1125@gmail.com}
\address{Bruno Sc\'ardua \\ Instituto de Matem\'atica - UFRJ, CP 68530. Rio de Janeiro - RJ, 21945-970 - Brazil}
\email{bruno.scardua@gmail.com}
\date{}
\begin{document}

\maketitle

%\tableofcontents
\section{Introduction}
\label{Section:Introduction}

In this paper we  study transversely holomorphic foliations of
complex codimension one with some hypothesis on the transverse
structure. A real codimension two smooth foliation $\fa$ of a
differentiable manifold $M^{\ell+2}$ is {\it transversely
holomorphic\/} of (complex) codimension one if its holonomy
pseudogroup is given by biholomorphic maps between open subsets of
$\bc$. In \cite{Brunella} and \cite{Ghys} one finds the complete
classification of transversely holomorphic flows on closed
3-manifolds. This paper deals with the case $\ell \ge 2$, i.e., the
leaves of $\fa$ have real dimension $\ell \ge 2$. Due to the lack of
some ingredients that play a fundamental role in the case of flows,
as harmonic time parametrizations and classification of compact
complex surfaces, we make additional hypothesis on the foliation
$\fa$. The idea is to classify, at a first moment, the simplest
transversely holomorphic foliations of codimension one. From the
structural point of view the simplest foliations are those with an
homogeneous transverse structure compatible with the transversely
holomorphic structure. These will be called {\it $\mathbb
C$-transversely homogeneous foliations}. Examples are given by
foliations with $\mathbb C$-additive, affine or projective
transverse structure. We shall prove that these are the only cases
(cf. Proposition~\ref{Proposition:homogeneous}).

\begin{theorem}
\label{Theorem:sheaf}

Let $\fa$ be a transversely holomorphic foliation of codimension one on $M$
given  by a
transversely holomorphic integrable $1$-form $\Om$. Assume that $\fa$ is
$\mathbb C$-transversely affine on $M\backslash\La$ for a compact leaf $\Lambda \subset M$ such that:
\begin{enumerate}

\item $\La$ contains an attractor on its holonomy
group.

\item We have $H^1(\Lambda, \mathbb C^*)=0$ or,
more generally, $\La$ is  given by some equation $\La\colon \{f=0\}$
where $f\colon M \to \bc$ is transversely holomorphic.

\item  $H^1(M,\Om_{\fa}^{\O}) = 0$ where
$\Om_{\fa}^{\O}$ is the sheaf of closed transversely holomorphic
$1$-forms on $M$ that vanish on $T\fa$.
\end{enumerate}
\noindent Then the holonomy group of $\Lambda$ is abelian. In particular  $\fa$ is locally logarithmic, i.e., given by a closed transversely meromorphic
$1$-form $\xi$  with simple poles in a neighborhood of $\La$ in $M$.
\end{theorem}

As an application of our techniques we prove:

\begin{theorem}
\label{Theorem:main}
 Let $\G$ be a codimension one holomorphic
foliation with singularities in a complex Stein manifold $X^n$ of
dimension $n \ge 3$. Let $A \subset X$ be a relatively compact
open subset with smooth simply-connected boundary $M = \po A$ transverse to $\G$. Suppose that the induced foliation $\fa =
\G|\_M$ is holomorphically transversely affine in $M\backslash
\La$ where $\La = \Ga \cap M$ and $\Ga \subset X$ is transversely
analytic and invariant by $\G$ {\rm(}e.g. $\Ga \subset X$ is
closed and invariant by $\G${\rm)}. Also assume that each leaf $L\subset \Lambda$ of
$\fa$  contains an attractor on its holonomy
group. Then $\G$ is logarithmic in a neighborhood of $\ov A$ in
$X$.
\end{theorem}

Theorem~\ref{Theorem:main} also holds under some more general conditions (see Remark~\ref{Remark:moregeneral}).

 Using Theorem~\ref{Theorem:main} and \cite{Scardua-Ito1} we promptly obtain:
\begin{Corollary}
\label{Corollary:sphere} Under the  hypotheses of
Theorem~\ref{Theorem:main} if $A$ is diffeomorphic to a ball $B^{2n}
\subset \bc^n$ and $M$ is diffeomorphic to a sphere $M \simeq
S^{2n-1} \subset \re^{2n}$ then $n=2$ and $\G|\_{\ov A}$ is
holomorphically conjugate to a linear foliation $\L\colon z\,dw -
\la w\,dz = 0$ on $\bc^2$ with $\la \in \bc \backslash \re$.
\end{Corollary}

\section{Transversely holomorphic foliations with homogeneous transverse structure}

Let $\fa$ be a transversely holomorphic foliation of codimension one on $M^{\ell +2}$.
Thus any point $p\in M$ has an open neighborhood $p \in U\subset M$ where we have local coordinates $(x,z)\in \mathbb R^\ell \times \mathbb C$ where $\fa$ is given by $z=c \in \mathbb C$.
We can, as in \cite{Brunella} introduce the sheaf $\mathcal O(\fa)$ (respectively $\mathcal M(\fa)$)
of {\it transversely holomorphic/meromorphic} functions on $M$ as
given by the  functions defined on open subsets of $M$ which are
locally constant along the leaves of $\fa$ and
transversely holomorphic. This means that such a function writes
in the local coordinates as $f(z)$ with $f(z)$ holomorphic. Similarly we can introduce the sheaf $\Omega^1(\fa)$ (respectively, $\Omega_m^1(\fa)$) of  transversely holomorphic/meromorphic one-forms on $M$.

\begin{Definition}
{\rm
A smooth foliation $\fa$ of real codimension two has a {\it holomorphic
homogeneous transverse structure\/} if there are a complex Lie
group $G$, a connected closed complex subgroup $H < G$ such that $\fa$
admits an atlas of submersions $y_j\colon U_j \subset M \to G/H$
satisfying $y_i = g_{ij}\circ y_j$ for some locally constant map
$g_{ij}\colon U_i \cap U_j \to G$ for each $U_i \cap U_j \ne
\emptyset$. In other words, $\fa$ has a  transversely holomorphic atlas of
submersions whose  transiction maps are given by left
translations on $G$ and submersions taking values on the
homogeneous space $G/H$. In particular $\fa$ is transversely holomorphic. We shall say that $\fa$ is $\mathbb C$-transversely
homogeneous {\it of model\/} $G/H$.
}
\end{Definition}

\vglue .1in The foliation is {\it $\mathbb C$-transversely additive}
when there are  maps $g_{ij}$ in the definition of holomorphic
homogeneous transverse structure which are of the form $g_{ij}(z) =
z+b_{ij}$\,, \, $b_{ij} \in \bc$ locally constant in $U_i \cap
U_j$\,. If $g_{ij}(z) = a_{ij}\,z + b_{ij}$\,, for locally constant
$a_{ij} \in \bc-\{0\}$ and $b_{ij} \in \bc$ we say that $\fa$ is
{\it $\mathbb C$-transversely affine\/} and it is {\it $\mathbb
C$-transversely projective\/} if $g_{ij}(z) = \frac{a_{ij}\,z +
b_{ij}}{c_{ij}\,z
+ d_{ij}}$ with locally constant $\begin{pmatrix} a_{ij} &b_{ij}\\
c_{ij} &d_{ij}\end{pmatrix} \in \SL(2,\bc)$.

Next we give a couple of examples of these structures.

\begin{Example}{\rm   We will
define a $\mathbb C$-transversely  affine (transversely holomorphic)
foliation on a compact manifold. This will be a non-singular
foliation with dense leaves which are biholomorphic to $\mathbb
C^*\times\mathbb C^*$ or cylinders $\mathbb C^*/\mathbb
Z\times\mathbb C^*$. We begin with a general construction inspired
in \cite{seke}. Let $M$ be a compact differentiable manifold of real
dimension $n$. Let $\omega$ be a closed 1-form on $M$ and $f\colon M
\to M$  a smooth diffeomorphism such that $f^*\omega = \la \omega$
for some $\lambda \in \mathbb C^*$ with $|\la| \ne 1$. Define $\Om$
on $M\times\mathbb C^*$ by $\Om(x,t) = t.\omega(x), \, x\in M, t \in
\mathbb C^*$. Then we have $d\Om = \eta\wedge\Om$ where $\eta(x,t)$
is defined by $\eta(x,t) = \frac{dt}{t}$\,. Then
 $d\eta = 0$ and $\eta$ is transversely holomorphic,
 thus $\Om$ defines a
codimension one transversely holomorphic foliation $\widetilde\fa$
in the product $M\times\mathbb C^*$. The foliation $\tilde \fa$ is
$\mathbb C$-transversely affine as a consequence of
Proposition~\ref{Proposition:forms}. Now we consider the action
$\Phi\colon \mathbb Z\times(M\times\mathbb C^*) \longrightarrow
M\times\mathbb C^*, n,(x,t) \longmapsto (f^n(x),\la^{-n}.t)$. This
is a locally free action generated by the transversely holomorphic
diffeomorphism $\vr\colon M\times\mathbb C^* \to M\times\mathbb
C^*$, $\vr(x,t) = (f(x),\la^{-1}\,t)$. Notice that the action of
$\vr$ preserves $\tilde \fa$ as well as its $\mathbb C$-affine
transverse structure. Indeed we have $\vr^*\,\Om(x,t) =
\la^{-1}\,t.\la\,w(x) = \Om(x,t)$ and $\vr^*\eta=\eta$.  In
particular,  $\widetilde\fa$ induces a codimension one transversely
holomorphic foliation $\fa$ on the quotient manifold $V =
(M\times\mathbb C^*)/\mathbb Z$. The foliation $\fa$ inherits a
$\mathbb C$-affine transverse  structure induced by the pair
$(\Om,\eta)$. This is a pretty general construction. Let us give a
more concrete one.

 We consider a variant of the Furness
example (see also \cite{Scardua1}): Consider the unimodular map $U = \begin{pmatrix} 1 &1\\
1 &2\end{pmatrix} : \mathbb R^2 \to \mathbb R^2; \quad U(x,y) =
(x+y, x+2y).$ This map induces a biholomorphism $f\colon T^2 \to
T^2$, where $T^2 =S^1 \times S^1$.  We consider $\tilde{\tilde
\omega} := (1+\sqrt 5)dx-2dy$ in $\mathbb R^2$. We then have
$U^*\tilde{\tilde \omega} = \lambda .\tilde{\tilde \omega}$ where
$\lambda = \frac{2}{3-\sqrt 5}$ and $U$ is $\mathbb Z\times\mathbb
Z$ invariant ($\mathbb Z\times\mathbb Z$ acts on $\mathbb R^ 2$ by
the natural product action) so that it induces a 1-form  $\omega$ in
the torus $T^2$. The 1-form $\omega$ satisfies $f^*\omega =
\lambda.\omega$. The foliation induced on $V = (M\times\mathbb
C^*)/\mathbb Z =(T^2\times\mathbb C^*)/\mathbb Z$ is $\mathbb
C$-transversely affine. }
\end{Example}

\begin{Example}
{\rm
We consider a {\it Riccati foliation} on $M^3=S^1 \times S^2$ as follows.
Recall that $S^1=\mathbb RP^1=\mathbb R \cup \{\infty\}$ and $S^2=\mathbb CP^1= \ov{\mathbb C}= \mathbb C \cup \{\infty\}$. Given affine coordinates $x\in \mathbb R\subset  S^1$ and $z\in \mathbb C\subset S^2$ we consider the differential equation
\[
\dot x= p(x), \, \dot z = a(x) z^2 + b(x) z + c(x)
\]
where $p(x), a(x), b(x), c(x) \in \mathbb R[x]$ are polynomials.
In real coordinates we may write $z=z_1 + i z_2$ where $z_1, z_2 \in \mathbb R$ and $i^2 =-1$ and rewrite the above ODE as
\[
\dot x= p(x), \, \dot z_1 = a(x) ( z_1 ^2 + z_ 2 ^2 ) + b(x) z_1 + c(x), \, \dot z_2 = 2 a(x) z_1 z_2 + b(x)z_2
\]
which defines a polynomial vector field $\mathcal Z$ on $\mathbb R \times \mathbb R^2 \subset M^3$. This polynomial vector field induces a transversely holomorphic foliation $\fa(\mathcal Z)$ on $M^3 \setminus S$ for a finite singular set $S\subset M^3$ given in the affine space by the pairs
$(x_j,z_j)$ such that $p(x_j)=0$ and $a(x_j) z_j ^2 + b(x_j) z_j +
c(x_j)=0$. The foliation $\fa(\mathcal Z)$ has the following property:

$\fa(\mathcal Z)$ is transverse to the fibers of the fibration
$S^1 \times S^3 \to S^1, ,\ (x,z) \mapsto x$, except for a finite number of such fibers (those given in the affine part by $p(x)=0$ and perhaps the fiber $x=\infty$).  This implies, together with Ehresmann theorem on fibrations \cite{Godbillon} that $\fa(\mathcal Z)$ is
transverse to the fibers of the fiber bundle $S^1 \times S^2 \to S^1$
except for those mentioned fibers. All these non-transverse fibers are invariant by $\fa(\mathcal Z)$. In particular, if $\Lambda \subset M^3$ denotes the union of these invariant fibers then the restriction $\fa(\mathcal Z)\big|_{M ^3 \setminus \Lambda}$ is conjugate to the suspension of a certain representation $\vr \colon \pi_1(S^1 \setminus \sigma) \to \Diff(S^2)$. Since the foliation is clearly transversely holomorphic, in a way compatible with the holomorphic structure
of the fibers of the bundle $M^3 \to S^1$, we conclude that the image $\vr(\pi_1(S^1 \setminus \sigma))\subset \Diff(S^2)$ consists of holomorphic diffeomorphisms of the Riemman sphere, i.e., we have a suspension of a group of Moebius maps $G\subset SL(2, \mathbb C)$.
This shows that $\fa(\mathcal Z)$ is indeed, transversely projective in $M\setminus \Lambda$.

}

\end{Example}

The proof of the next proposition follows the one in
\cite{Scardua1}.
\begin{Proposition}
\label{Proposition:homogeneous}
 Let $\fa$ be a codimension one
transversely holomorphic foliation on $M$. If $\fa$ is
holomorphically transversely homogeneous then $\fa$ is
holomorphically transversely additive, affine or projective.
\end{Proposition}

\begin{proof}
By the hypotheses the quotient  $R=G/ H$ is a Riemman surface. We may assume that $R$ is simply-connected (otherwise we consider the universal covering of $R$ and lift the submersions to this space).  By the Riemann-Koebe Uniformization theorem we have a conformal equivalence  $R \equiv \ov{\mathbb C}, \mathbb C$ or $\mathbb D$ the unitary disc. This implies that
either $G\subset \Aut(\ov{\mathbb C})=\mathbb P SL(2,\mathbb C), G\subset \Aut(\mathbb C)=\Aff(\mathbb C)$ or
$G\subset \Aut(\mathbb D)\cong \mathbb P SL(2,\mathbb R)$.  The proposition follows.
\end{proof}

\section{The transversely additive case}

In what follows we study the following situation: $\fa$ has  a
$\mathbb C$-additive transverse structure on
$M\backslash\La$ where $\La \subset M$ is a finite union of
compact leaves of $\fa$. This means that there is an open cover of $M\setminus \Lambda$ by open sets $U_j$ where are defined submersions $y_j \colon U_j \to \mathbb C$ such that $\fa\big|_{U_j}$ is given by $dy_j=0$ and for each intersection $U_i \cap U_j$ we have
$y_i = y_j + a_{ij}$ for some locally constant $a_{ij}$.
Taking the differential $dy_i = dy_j$ we conclude that
there exists a closed transversely
holomorphic 1-form $\omega$ on $M\backslash\La$ such that
$\fa|\_{M\backslash\La}$ is given by $\omega=0$ and, by
construction, $\omega\big|_{U_j}=dy_j$. Thus $\omega$ is invariant by the holonomy pseudogroup of
$\fa$ in $M\backslash\La$.
 Our aim is to study the
holonomy group of a leaf $L \in \fa$, \, $L \subset \La$.

\begin{Proposition}Given any leaf $L \subset \La$ the holonomy
group of $L$ is solvable.
\end{Proposition}
\begin{proof} Given a transverse disc $D \subset M$, \, $D
\cap \La = D \cap L = \{p\}$, \, $D\pitchfork\fa$ we consider the
holonomy representation $\Hol(\fa,L,D,p) \cong G < \Diff(\bc,0)$
as a subgroup of the group of germs of holomorphic diffeomorphisms
fixing the origin $0 \in \bc$.  Assume by contradiction that
$\Hol(\fa,L,D,p)$ is not solvable. By \cite{Nakai},\cite{Loray2}
 we may find in $D$ a dense subset of points $q \in D$ which
are hyperbolic fixed points for elements $f_q \in
\Hol(\fa,L,D,p)$. Given such a fixed point $f_q(q)=q$ we choose a
local chart $z\colon V \subset D \to \bc$, taking $q$ onto $0 \in
\bc$ and such that $f_q(z) = \la z$ with $|\la| \ne 1$ is linear.
The leaf $L_q \in \fa$ that contains $q$ has therefore an homotopy
class $\ga \in \pi_1(L_q,q)$ that originates the holonomy map
$f_\ga = f_q|\_V \in \Hol(\fa,L_q,V,q)$.

\par As we have seen, there exists a closed transversely
holomorphic 1-form $\omega$ on $M\backslash\La$ such that
$\fa|\_{M\backslash\La}$ is given by $\omega=0$ and, by
construction, $\omega$ is invariant by the holonomy pseudogroup of
$\fa$ in $M\backslash\La$. In particular, we must have
$f_\ga^*(\omega|\_V) = \omega|\_V$\,. We write $\omega|\_V =
g(z)\,dz$ to obtain $f_\ga^*(\omega|\_V) = \omega|\_V \Rightarrow
(*) \la g(\la z) = g(z) \Rightarrow g(z) = \frac{1}{z}g_{-1}$ for
some constant $g_{-1} \in \bc^*$: indeed, write $g(z) =
\sum\limits_{j=0}^{+\infty} g_j\,z^j$ in Laurent series, then (*) is
equivalent to $\sum\limits_{j=-\infty}^{+\infty} g_j\,\la^{j+1}z^j =
\sum\limits_{j=-\infty}^{+\infty} g_zz^j \Rightarrow
(\la^{j+1}-1)\cdot g_j=0$, \, $\forall\, z \in \bz \Rightarrow
g_j=0$, \, $\forall\,j \ne -1$.  Thus $\omega|\_V =
g_{-1}\frac{dz}{z}$ and therefore, since the set of hyperbolic fixed
points is dense in $D$ we must have $g_{-1} = 0$ and $\omega|\_V
\equiv 0$. This implies $\omega\equiv 0$ in $M\backslash \La$,
contradiction. \end{proof}

If we add a dynamical hypothesis, then more can be said:

\begin{Proposition}
\label{Proposition:logarithmic}
Let $\fa$ be a transversely
holomorphic codimension one foliation on $M$ and suppose that
$\fa$ is holomorphically transversely additive in $M\backslash\La$
where $\La \subset M$ is a finite union of compact leaves. Then
$\fa$ is given by closed transversely meromorphic $1$-form
$\omega$ with simple poles and polar divisor $(\omega)_\infty =
\La$ provided that each leaf $L \subset \La$ contains an attractor
on its holonomy group.
\end{Proposition}

The proof is a straightforward consequence of the following lemma:

\begin{Lemma}[extension lemma]
\label{Lemma:extension} If $\omega$ is a closed transversely
holomorphic $1$-form defining $\fa$ in $W^* = W\backslash\La$ then
$\omega$ extends to a transversely meromorphic $1$-form in $W$
provided that the holonomy group $\Hol(\fa,L)$ contains an attractor
for each leaf $L \subset \La$.
\end{Lemma}
\begin{proof} We may assume that $\La=L$ is a single
compact leaf of $\fa$. Choose a $\C^\infty$ tubular neighborhood
$\U$ of $\La$ fibered by holonomy holomorphic discs and we may
assume that $\U = W$; the extension is a local problem around
$\La$. Let $\pi\colon \widetilde W^* \to W^*$ be the universal
covering of $W^*$ with transversely holomorphic projection. Lift
$\omega$ to a $\Delta$-form $\tilde \omega = \pi^*(\omega)$ in
$\widetilde W^*$. Then $\tilde w$ is closed, transversely
holomorphic therefore $\tilde \omega = d\widetilde F$ for some
function $\widetilde F\colon \widetilde W^* \to \bc$ which is
transversely holomorphic.

\noindent Thus $\omega = dF$ for a {\it multivalued\/}
transversely holomorphic function $F$ in $W^*$ (notice that $F$ is
not actually a function in $W^*$). Given a point $p \in \La$ we
consider the corresponding disc $D_p \ni p$ and its punctured disc
$D_p^* = D_p-\{p\} \subset W^*$. We have $D_p^* \simeq \bd^* = \{z
\in \bc; 0 < |z| < 1\}$ by conformal equivalence. The restriction
$\pi|\_{\pi^{-1}(D_p^*)}\colon \pi^{-1}(D_p^*) \to D_p^*$ is
therefore a holomorphic covering of the punctured unit disc
$\bd^*$. We may therefore consider a holomorphic covering
$\Pi\colon \bd \to \bd^*$, a lift $\tilde f\colon \bd \to \bd$ of
the restriction to $\bd^*$ of the hyperbolic holonomy
diffeomorphism $f\colon \bd \to \bd$, \, $f(z) = \la z$ the
restrictions $\omega|\_{\bd^*}$ and $\tilde \omega|\_{\bd} =
d(F|_\bd)$.

\noindent The lift $\tilde f$ preserves the foliation
$\widetilde\fa = \pi^*(\fa)$ and therefore it satisfies
$\widetilde F(\tilde f(\tilde z)) = \widetilde F(\tilde z)$, \,
$\forall\,\tilde z \in \bd$. We have $\Pi(\tilde z) = e^{2\pi
i\tilde z} = z$ so that if $\la = e^{2\pi i\mu}$ then $\tilde
f(\tilde z) = \mu+\tilde z$ we may assume. Therefore $\widetilde
F(\tilde z)$ satisfies
$$
(\star)\quad
\begin{cases}
d\widetilde F(\tilde z+1) = d\widetilde F(\tilde z)\\
\widetilde F(\tilde z + \mu) = \widetilde F(\tilde z)
\end{cases}
$$
Since $|\la| \ne 1$ we have $\mu \in \re$ and therefore
$d\widetilde F(\tilde z)$ gives a holomorphic differential 1-form
in the complex 1-torus $\bc\big/(\bz\oplus\mu\bz)$, therefore
$d\widetilde F(\tilde z)$ must be linear, i.e.,

$d\widetilde F(\tilde z) = \al\,d\tilde z$ for some $\al \in
\bc-\{0\}$. This gives $\tilde \omega = \al\,\frac{d(\log z)}{2\pi
i} = \al(2\pi i)^{-1}\, \frac{dz}{z}$ and therefore $\omega(z) =
c\cdot\frac{dz}{z}$ for some constant $c \in \bc-\{0\}$.

\par This proves the extension of $\omega|\_{\bd^*}$ to $\bd$. By
Hartogs' Extension Theorem $\omega$ extends to a transversely
meromorphic 1-form on $W$ with polar set $(\omega)_\infty = \La$
of order one.

\end{proof}

Lemma~\ref{Lemma:extension} will be useful also in the $\mathbb
C$-affine case. For the moment we state a straightforward
consequence of the proof above:

\begin{Corollary}
If $\fa$ is transversely holomorphically additive in
$M\backslash\La$ as in Proposition\,\ref{Proposition:logarithmic}
and the leaf $L \subset \La$ contains some hyperbolic attractor as
its holonomy group then $\Hol(\fa,L)$ is abelian linearizable.
\end{Corollary}

\noindent Notice that this is not clear at first sight and shows
that an additive transverse holomorphic structure can
``degenerate'' into a multiplicative structure.

\section{The transversely affine case}

Let $\fa$ be a real codimension $2$  foliation on $M^{\ell+2}$. The foliation $\fa$ is {\it $\mathbb C$-transversely affine} if there exists a transversely holomorphic
atlas of submersions $y_j\colon U_j \to \bc$ for $\fa$ such that
if $U_i \cap U_j \ne \emptyset$ then $y_i = a_{ij}\,y_j + b_{ij}$
for locally constant maps $a_{ij}, \, b_{ij} \colon U_i \cap U_j
\to \bc$.

The following proposition is a characterization of the existence of
such structure in terms of differential forms.
\begin{Proposition}
\label{Proposition:forms}
The possible $\mathbb C$-affine
transverse structures for $\fa$ in $M$ are classified by the
collections $(\Om_j, \eta_j)$ of differential 1-forms defined in
the open sets $U_j \subset M$ such that:

{\rm(i)}\,\,$\Om_j$ and $\eta_j$ are transversely holomorphic,
$\Om_j$ is integrable and defines $\fa$ in $U_j$\,, \,\, $d\Om_j =
\eta_j \wedge \Om_j$ and $d\eta_j=0$ in $U_j$\,, if $U_i \cap U_j
\ne \emptyset$ then $\Om_i = g_{ij}\,\Om_j$ and $\eta_i = \eta_j +
\frac{dg_{ij}}{g_{ij}}$ for non-vanishing transversely holomorphic
function $g_{ij}\colon U_i \cap U_j \to \bc-\{0\}$.

{\rm(ii)} Two such collections $(\Om_j, \eta_j)$ and
$(\Om_j',\eta_j')$ define the same affine transverse structure for
$\fa$ in $M$ if and only if \, $\Om_j' = g_j\,\Om_j$ and $\eta_j' =
\eta_j + \frac{dg_j}{g_j}$ for some transversely holomorphic
non-vanishing functions $g_j\colon U_j \to \bc-\{0\}$.
\end{Proposition}

\begin{proof} First we prove (i); assume that $\fa$ is
$\mathbb C$-transversely affine with transversely holomorphic atlas
of submersions $y_j\colon U_j \to \bc$. Given any transversely
holomorphic non-singular 1-form $\Om_j$ defining $\fa$ in $U_j$ we
have $\Om_j = g_j\,dy_j$ for some transversely holomorphic function
$g_j\colon U_j \to \bc-\{0\}$ and we define $\eta_j =
\frac{dg_j}{g_j}\,\cdot$ If $U_i \cap U_j \ne \emptyset$ then $\Om_i
= g_{ij}\,\Om_j$ and $y_i = a_{ij}\,y_j + b_{ij}$ imply $dy_i =
a_{ij}\,dy_j$ and therefore $a_{ij}\,g_i = g_j\,g_{ij}$\,. Thus
$\frac{dg_i}{g_i} = \frac{dg_j}{g_j} + \frac{dg_{ij}}{g_{ij}}$ in
$U_i \cap U_j$\,. Clearly $d\eta_j=0$, \, $d\Om_j = \eta_j \wedge
\Om_j$ and $\eta_i = \eta_j + \frac{dg_{ij}}{g_{ij}}\,\cdot$ This
proves (i). Item (ii) is proved similarly and we refer to
\cite{Scardua1}.

\end{proof}

\subsection{Holonomy}
From now on in this section we consider the following situation.
We have $\fa$ a transversely holomorphic foliation on $M$, having a
$\mathbb C$-affine transverse structure in $M\setminus \Lambda$ for some compact analytic set of dimension one $\Lambda \subset M$.
We shall study  the holonomy groups of leaves $L
\subset \La$. We have seen that in the $\mathbb C$-additive case, the holonomy groups are solvable. A similar result holds for the $\mathbb C$-affine case. Nevertheless, the proof  requires more technical features as we
will see below.

\begin{Theorem}
\label{Theorem:holonomymonodromy} Let $\fa$ be a transversely
holomorphic foliation of codimension one on $M^{\ell +2}$. Assume
that $\fa$ has a $\mathbb C$-transversely affine foliation on
$M^{\ell + 2}\setminus \Lambda$ for a compact curve $\Lambda \subset
M$. Given any leaf $L \subset \La$ the holonomy group of $L$ is
solvable.
\end{Theorem}

\begin{proof} Given a point $p \in M\backslash\La$ we may
consider a small open simply-connected neighborhood $U_p \ni p$ in
$M\backslash\La$ where we can write $\eta|\_{U_p} =
\frac{dg_p}{g_p}$ for some transversely holomorphic non-vanishing
function $g_p\colon U_p \to \bc-\{0\}$ and then $0 =
d\left(\frac{\Om}{e^{\int\eta}}\right) =
d\left(\frac{\Om}{g_p}\right)$ in $U_p$ implies $\Om = g_p\,dF_p$
for some transversely holomorphic function $F_p\colon U_p \to \bc$,
which is a local first integral for $\fa$ in $U_p$\,. If we choose
another $\bar p \in U_{\bar p} \subset M\backslash\La$ then in case
$U_p \cap U_{\bar p} \ne 0$ we have $F_{\bar p} = \al\,F_p + \be$
for some affine mapping $(z\mapsto \al z+\be)$. We introduce
therefore the multiform function $F$ on $M\backslash\La$ by writing
$F = \displaystyle\int \frac{\Om}{e^{\int\eta}}$ or also $dF =
\frac{\Om}{e^{\int\eta}}\,\cdot$ This is a locally well-defined
transversely holomorphic function $\left(F|\_{U_p} = F_p \text{ as
above}\right)$ which lifts to well-defined transversely holomorphic
function on the universal covering $\widetilde{M\backslash\La}$ of
$M\backslash\La$. Fixed a point $p \in M\backslash\La$ and any local
determination $F_p$ of $F$ and a path $\ga\colon [0,1] \to
M\backslash\La$ we may consider the transversely analytic
continuation $F_{p,\ga}$ of $F_p$ along $\ga$. If $\ga$ is closed,
$\ga(0) = \ga(1)$ then $F_{p,\ga}$ depends only on the homotopy
class $[\ga] \in \pi_1(M\backslash\La;p)$. Let $A(F)$ given by $A(F)
:= \{F_{p,[\ga]}; [\ga] \in \pi_1(M\backslash\La;p)\}$ be the set of
such transversely analytic continuations and $P^{-1}(p) =
\{\text{determinations } F_p \text{ on } F \text{ at } p\}$. Then
$\pi_1(M\backslash\La;p)$ acts transitively on $P^{-1}(p)$. There
exists therefore a regular covering $A(F)
\underset{P}{\longrightarrow} M\backslash\La$ with total space
$A(F)$, fiber over $p$ equal to $P^{-1}(p)$ and covering
automorphisms group isomorphic to
$\pi_1(M\backslash\La;p)\big/P_\#(\pi_1(A(F),F_p) =: \M(F)$. We call
$\M(F)$ the {\it monodromy group} of $F$.

\noindent The canonical projection $\mu\colon
\pi_1(M\backslash\La;p) \to \M(F)$ is called the {\it monodromy
map\/} of $F$ and associates to each homotopy class $[\ga] \in
\pi_1(M\backslash\La/p)$ the corresponding determination
$F_{p,[\ga]}$ at the point $\ga(1) = p$.

Given any leaf $L_0 \subset \La$ we fix a point $p_0 \in L_0$ and
a small holomorphic holonomy disc $D_0$\,, \, $D_o \pitchfork
\fa$, \, $D_0 \cap L_0 = \{p_0\}$ so that we have a holonomy
representation $\Hol(\fa, L_0, D_0, p_0) \subseteq
\Diff(D_0;p_0)$. Given a $C^\infty$ tubular neighborhood $r\colon
N \to L_0$ fibered by holomorphic holonomy discs $D_q =
r^{-1}(q)$, \, $q \in L_0$ with $r^{-1}(p_0) = D_0$ and $N^* =
N\backslash L = N\backslash\La$ we obtain by restriction a
$C^\infty$ fibration $r^*\colon N^* \to L_0$ fiber a punctured
disc $\bd^* = \bd-\{0\}$. The homotopy sequence for this fibration
gives the exact sequence below
$$
\begin{matrix}
0 &\longrightarrow &\bz &\longrightarrow& \pi_1(N^*,\tilde p_0) \longrightarrow \pi_1(L_0;p) \to 0\\
&\searrow &\pi_1(\bd^*) &\nearrow&
\end{matrix}
$$
where $\tilde p_0 \in N^*$ is a fixed base point with $r(\tilde
p_0)=p_0$. Denote by $A(F)|\_{N^*}
\underset{P|_{N^*}}{\longrightarrow} N^*$ the natural restriction
of $A(F) \underset{P}{\longrightarrow} M\backslash\La$ and by
$A(F)|\_{N^*}$ the connected component which contains the local
determination $F_{\tilde p_0}$\,.  As above the monodromy map we
denote by
$$
\mu\colon \pi_1(N^*,\tilde p_0) \to \M(F)(N^*) \cong
\frac{\pi_1(N^*,\tilde p_0)}{P_\#(\pi_1(A(F)|\_{N^*};F_{\tilde
p_0})}\,\cdot
$$
We shall use the following lemma

\begin{Lemma}[\cite{Scardua2}] There exists a unique morphism $(\mu)$ which
makes commutative the following diagram
$$
\begin{matrix}
0 \longrightarrow \bz \longrightarrow &\pi_1(N^*,\tilde p_0)& \longrightarrow &\pi_1(L_0,p_0)& \to 0\\
&\downarrow \mu&  &\downarrow(\mu)&\\
&\M(F;N^*)& \longrightarrow &\frac{\M(F;N^*)}{\bz}& \to 0.
\end{matrix}
$$
\end{Lemma}

We finally define $\mu(F;L_0) := \frac{\M(F;N^*)}{\bz}$ as the
monodromy of $F$ associated to $L_0$ and call the morphism
$(\mu)\colon \pi_1(L_0;p) \to \M(F; L_0)$ the monodromy mapping of
$F$ relatively to $L_0$\,. Now we use

\begin{Lemma}[\cite{Scardua2}]
There exists a surjective morphism
$\al$ which makes commutative the diagram
$$
\begin{matrix}
&\pi_1(L_0;p_0)& \\
\Hol  \swarrow & & \searrow (\mu)\\
\Hol(\fa,L_0,D_0) &\overset{\al}{\longrightarrow}& \M(F;L_0)
\end{matrix}
$$
\end{Lemma}

\noindent Given any local determination $\ell(y)$ of
$F|\_{D_0^*}$\,, where $y$ is a holomorphic coordinate on $D_0$\,,
we denote by $\ell(y)_{[\ga]}$ its analytic continuation along
$[\ga] \in \pi_1(L_0;p_0)$ so that we have an element
$\ell(y)_{[\ga]} = (\al\circ\Hol)([\ga])$ of the monodromy group
$\M(F;L_0)$. Also we need:

\begin{Lemma}[\cite{Scardua2}]
For each $[\ga] \in \pi_1(L_0,p_0)$ we have $\ell(y)_{[\ga]} =
a_{[\ga]} y + b_{[\ga]}$ for some affine mapping $(z \mapsto
a_{[\ga]}z + b_{[\ga]})$. In particular $\M(F_0;L_0)$ is solvable.
\end{Lemma}
Since the exact sequence $0 \to \Ker \al \to \Hol(\fa,L_0,D_0)
\overset{\al}{\longrightarrow} \M(F;D_0^*) \to 0$ has
$\M(F;D_0^*)$ solvable and $\Ker(\al)$ abelian we conclude that
$\Hol(\fa,L_0,D_0)$ is solvable.

\end{proof}

\subsection{Extension}
We keep on considering the situation of the previous paragraph.
Moreover, let $\fa$ be  given in $M$ by a
transversely holomorphic integrable $1$-form $\Om$.
Our first step is the following:

\begin{Proposition}
Let $L \subset \La$ a leaf whose holonomy group contains a
hyperbolic attractor. Then we may find a fibered neighborhood $W$
of $L$ by holonomy discs $D$ as above and a transversely
meromorphic $1$-form $\eta_L$ in $W$ such that:

\noindent{\rm(i)} $\eta_L|\_D$ is a meromorphic $1$-form writing as
$\eta_L|\_D(y) = \frac{ady}{y} + \frac{dg}{g}$\, $a \in \bc$ in a
suitable holonomy holomorphic coordinate $y$ in $D$ for which $\Om =
gdy$.

\noindent{\rm(ii)} If $L$ has abelian holonomy we may take $a=0$
and if $L$ is non abelian $a=k+1$ for the imbedding of
$\Hol(\fa,L,D)$ in $\bh_k$\,.
\end{Proposition}

\begin{proof}
We already know from Theorem~\ref{Theorem:holonomymonodromy} that
the holonomy group of each leaf $L\subset\Lambda$ is solvable. Given
a point $p \in L$ consider a transverse disc $\Sigma\subset M$ to
$\fa$ with $\Sigma \cap \Lambda = \Sigma \cap L =\{p\}$. We put
$G=\Hol(\fa,L,\Sigma,p)$. By means of a local holomorphic coordinate
$z \in \Sigma$ we may consider $G$ as a subgroup of $\Diff(\mathbb
C,0)$.
\begin{Claim}
There is a germ of a holomorphic vector field  $\mathcal X(z)$ in
$\Sigma$ such that  for any $g \in G$ we have $g_*\mathcal X =
c_g\mathcal X$ for some constant $c_g \in \bc^* = \bc-\{0\}$.

\end{Claim}
Indeed, it is  well-known that a solvable  subgroup $G <
\Diff(\bc,0)$ admits such a formal vector field which is
projectively invariant. Moreover, this vector field is is convergent
in the case the group contains some nonresonant (hyperbolic for
instance) map \cite{Cerveau-Moussu}, \cite{Loray}). Now we may
write, up to an analytic change of coordinates in $\Sigma$,
$\mathcal X(z) = \frac{z^{k+1}}{1+a z^k}\,\frac{d}{d z}$. For any $g
\in G$ we have $g_*(\frac{z^{k+1}}{1+a z^k}\,\frac{d}{d z}) =
c_g(\frac{z^{k+1}}{1+a z^k}\,\frac{d}{d z})$ for some constant $c_g
\in \bc^* = \bc-\{0\}$. We have two cases to consider: \vglue.1in
\noindent{\bf $G$ is abelian}. In this case, because it contains a
hyerbolic (analytically linearizable) map, the group $G$ is
analytically linearizable. We may therefore construct a closed
meromorphic one-form $\omega$ with simple poles,
$(\omega)_\infty=L$, in a neighborhood $W$ of $L$ in $M$. The form
$\omega$ is given in any linearizing transverse coordinate $z$ by
$\omega(z) = \frac{dz}{z}$. In this case we have $\Omega = g \omega$
for some transversely meromorphic one-form $g$ in $W$. Since
$\Omega/g$ is closed, we have that $\eta_L:=\frac{dg}{g}$ satisfies
$d\Omega= \eta \wedge \Omega$.

\vglue.1in

\noindent{\bf  $G$ is solvable non-abelian}. In this case there
exists $g \in G$ with $c_g \ne 1$ and therefore $a=0$, i.e., $\chi =
z^{k+1}\,\frac{d}{d z}$ giving an analytic embedding $G
\hookrightarrow \bh_k= \{(z \mapsto \frac{\la z}{\sqrt[k]{1+\mu
z^k}})\}$\,. In this case we cover a neighborhood of $L\subset M$ by
local charts $(x,z)\in \mathbb R^\ell \times \mathbb C$ where $\fa$
is given by $dz=0$, and the coordinate $z$ gives the embedding of
the holonomy group as a subgroup of $\bh_k$. If we write on each
chart $\Omega(x,z)= gdz$ then we put $\eta(x,z):=\frac{dg}{g} +
(k+1) \frac{dz}{z}$. On each intersection of two such coordinate
charts $(x,z)$ and $(\tilde x, \tilde z)$, we have: (i)
$\Omega=gdz=\tilde g d \tilde z$ and  (ii) $\tilde z = \frac{\la
z}{\sqrt[k]{1+\mu z^k}}$.

Taking derivatives in the second equation we have
$\frac{d\tilde z}{{\tilde z}^{k+1}}= \frac{dz}{ \lambda^k z^{k+1}}$.
Replacing this in the first equation we conclude that $\tilde g {\tilde z}^{k+1}= \lambda ^k g z^{k+1}$. This implies that
$d \ln \tilde g {\tilde z}^{k+1}= d \ln g z^{k+1}$, i.e.,
$(k+1) \frac{d \tilde z}{\tilde z} + \frac{d \tilde g}{\tilde g} =
(k+1) \frac{dz}{z} + \frac{dg}{g}$. This defines the 1-form $\eta_L$ by $\eta_L(x,z):=(k+1) \frac{dz}{z} + \frac{dg}{g}$ on each coordinate system. This form satisfies the equation $d\Omega= \eta_L \wedge \Omega$.

\end{proof}

Then we prove the following extension result:

\begin{Proposition}
\label{Proposition:extension}
 Let $\fa$ be $\mathbb C$-transversely affine on $M\backslash\La$ and given in $M$ by  $\Om$. Suppose that
each leaf $L \subset \La$ contains an attractor on its holonomy
group then there is a closed transversely meromorphic $1$-form
$\eta$ in $M$ with polar set $(\eta)_\infty = \La$ of order one
and such that $d\Om = \eta \wedge \Om$. Moreover for any leaf $L
\subset \La$ we have:
\begin{itemize}
\item[{\rm (i)}] If $\Res_L\eta = a \notin \{2,3,4,\dots\}$
then $\Hol(\fa,L)$ is abelian linearizable.
\item[{\rm (ii)}] If $\Hol(\fa,L)$ is not abelian linearizable
then $\Res_L\eta = k+1$ with $k \in \bn$ and $\Hol(\fa,L)$ embeds
analytically into $\bh_k = \{(z \mapsto \frac{\la z}{\sqrt[k]{1+\mu
z^k}})\}$.
\end{itemize}
\end{Proposition}

\begin{proof}
The existence of a $\mathbb C$-affine  structure in $M\setminus
\Lambda$ gives us via Proposition~\ref{Proposition:forms}  a form
$\eta$ defined in $M \setminus \Lambda$ with the properties
announced. The problem is to show the extension of $\eta$ to
$\Lambda$. For this sake we consider the restriction of $\eta$ to
$W\backslash\La = W\backslash L =: W^*$. In $W^*$ we have $d\Om =
\eta \wedge \Om = \eta_L \wedge \Om \Rightarrow \eta-\eta_L =  h
\cdot\Om$ for some transversely meromorphic function $h$ in $W^*$
such that $d(h\Om)=0$.

\noindent If $\eta-\eta_L \not\equiv 0$ then $\Om|\_{W^*}$
admits $ h$ as an integrating factor. Let us prove the extension of  $\eta$
to $L$.  This is done  as follows:

 Now we have two possibilities:

\noindent{\bf $1^{\text{st}}$.} \,\, $\Hol(\fa,L)$ is abelian. In
this case, since it contains a linearizable attractor, $\Hol(\fa,L)$
is abelian linearizable and $\fa$ is given by a closed transversely
meromorphic 1-form $\xi$ in a fibered neighborhood $W$ of $L$ in $M$
as above, moreover $\xi|\_D$ writes as $\xi|\_D(z) = \frac{dz}{z}$
in suitable holonomy holomorphic coordinates $z$ in $D$. Since,
because $\Hol(\fa,L)$ contains a hyperbolic element, there exists no
meromorphic first integral for $\fa$ in $W$ we must have $\xi$
unique, up to multiplicative constants. Therefore $\omega(z) =
\vr_{-1}\,\frac{dz}{z}$ and $\eta|\_D(z) = \eta_L|\_D(z) +
\vr_{-1}\,\frac{dz}{z} = \eta_L|\_D(z) + \vr_{-1}\,\xi|\_D(z)$.
Therefore, $\eta = \eta_L + \text{ cte. } \xi$ in $W$.

\vglue.1in

\noindent{\bf $2^{\text{nd}}$.}\,\, $\Hol(\fa,L)$ is solvable but
not abelian. In this case there exists a holomorphic imbedding
$\Hol(\fa,L,D) \hookrightarrow \bh_k$ for a unique $k \in \bn$ and
$\fa|\_W$ cannot be given by a closed transversely meromorphic
1-form. Therefore $\omega\equiv 0$ and $\eta = \eta_L$ in $W$.

\end{proof}

\section{Proof of Theorem~\ref{Theorem:sheaf}}

 Let $\fa$ be $\mathbb C$-transversely affine on $M\backslash\La$ and given in $M$ by  $\Om$. Suppose that
each leaf $L \subset \La$ contains an attractor on its holonomy
group. By Proposition~\ref{Proposition:extension}  there is a closed transversely meromorphic $1$-form
$\eta$ in $M$ with polar set $(\eta)_\infty = \La$ of order one
and such that $d\Om = \eta \wedge \Om$. Now we consider the  special case where  $\La=L$ is a
single leaf of $\fa$.
Let $a \in \bc$ denote the residue
$\Res_L\,\eta$.
Let us prove that the holonomy group  $\Hol(\fa,L)$ is  abelian, which implies Theorem~\ref{Theorem:sheaf}.
If this is not the case we have  $a=k+1$ for some $k \in \bn$.
This case is dealt with via the following lemma:
\begin{Lemma}
\label{Lemma:logarithmic} Suppose that any closed transversely
holomorphic $1$-form on $M$ is exact and $\La = L = \{f=0\}$ for
some transversely holomorphic function $f\colon M \to \bc$. Then
$\fa$ is given in a neighborhood of $\La$ by a closed transversely
meromorphic $1$-form with simple poles $\omega$, defined in a
neighborhood $U \supset \La$ of $\La$ in $M$.
\end{Lemma}

\begin{proof} We shall first consider the following simpler case.

Assume that $\La$ is  given by some analytic equation $\La\colon
\{f=0\}$ where $f\colon M \to \bc$ is transversely holomorphic.

We may write by hypothesis $\eta = (k+1)\,\frac{df}{f} + d\vr$ for
some transversely holomorphic function $\vr\colon M \to \bc$. Thus
$d\Om = \eta \wedge \Om$ implies $d(\frac{\Om}{f^{k+1}e^\vr}) = 0$.

Thus, $\omega = \frac{\Om}{f^{k+1}e^\vr}$ is closed transversely
meromorphic and defines $\fa$ in $M$ with polar set $(\omega)_\infty
= \La$. But this implies that $\Hol(\fa,\La)$ is abelian and
therefore abelian linearizable for it contains an attractor. Using
well-known techniques from \cite{Camacho-LinsNeto-Sad},
\cite{Scardua2}, \cite{Cerveau-Mattei} one can construct a
transversely meromorphic 1-form $\omega$, with order one polar set
$(\omega)_\infty = \La$, is a neighborhood $U$ of $\La$ in $M$, such
that $\fa|\_U$ is given by $\omega$: given any holonomy disc $D$
with $D \cap \La \ne \emptyset$ and any holomorphic coordinate $z
\in D$ that linearizes $\Hol(\fa,\La,D)$ we define $\omega|\_D :=
\frac{dz}{z}$ and extend it constant along the leaves of the
foliation.

Now we shall consider the general case. In this case there is an
open cover of a neighborhood $W(\Lambda)$ of $\Lambda$ in $M$ by
open sets $U_j\subset M$ such that on each $U_j$ there is a
transversely holomorphic function $f_j \colon U_j \to \mathbb C$,
reduced and such that $\Lambda \cap U_j =\{f_j=0\}$. Then from the
above considerations we can write $\eta\big|_{U_j} =
(k+1)\,\frac{df_j}{f_j} + d\vr_j$ for some transversely holomorphic
function $\vr_j\colon U_j \to \bc$. Then we have a closed
transversely meromorphic one-form
$\omega_j:=\frac{\Om}{f_j^{k+1}e^{\vr_j}}$ defining $\fa$ in $U_j$.
Now observe that for each non-empty intersection $U_i \cap U_j \ne
\emptyset$ we have $(k+1)\,\frac{df_j}{f_j} + d\vr_j=
(k+1)\,\frac{df_j}{f_j} + d\vr_i$ and therefore $f_j
e^{\frac{1}{k+1} \vr_j} = c_{ij} . f_i e^{\frac{1}{k+1} \vr_i}$ for
some constant $c_{ij}\in \mathbb C$. From the hypothesis
$H^1(\Lambda, \mathbb C^*)=0$ we obtain constants $\{c_j\}$ such
that if $U_i \cap U_j \ne \emptyset$ then we have $c_{ij}=c_i/c_j$.
Thus we may define $f\colon W(\Lambda) \to \mathbb C$ by setting
$f\big|_{U_j}=c_j f_j e^{\frac{1}{k+1} \vr_j}$. This is a
transversely holomorphic function such that $\Lambda=\{f=0\}$. The
proof then follows from the first case.
\end{proof}

\section{Applications: Proof of Theorem~\ref{Theorem:main}}

In this section we give some applications of the techniques we have
introduced. In this course we shall prove
Theorem~\ref{Theorem:main}. We start with  the following situation:
$\G$ is holomorphic foliation of codimension one and with
singularities on a complex manifold $X$ of dimension $n \ge 1$.
Suppose that $\G$ is transversely affine (as a singular holomorphic
foliation) (see \cite{Scardua1},\cite{Scardua2}) on $X\backslash\Ga$
for some analytic invariant codimension one subset $\Ga \subset X$.
Let $M \subset X$ be a real closed hypersurface that we suppose to
be transverse to $\G$ (see \cite{Scardua-Ito2}). Then the induced
foliation $\fa = \G|\_M$ is transversely holomorphic of codimension
one with a holomorphic affine transverse structure on
$M\backslash\La$, \, $\La := \Ga\cap M$ is a union of compact leaves
of $\fa$.

\begin{Lemma}
\label{Lemma:tubular}
In the above situation for $\fa$, $\La$,
$\G$, $\Ga$, $M$ and $X$ suppose that each leaf $L \subset \La$
contains an attractor on its holonomy group. Then $\G$ admits in a
small tubular neighborhood $W$ of $M$ in $X$, a $1$-form $\eta$
closed meromorphic in $W$ with simple poles $(\eta)_\infty = \Ga
\cap W$ and such that $d\Om = \eta\wedge\Om$ where $\Om$ is any
holomorphic $1$-form that defines $\G$ in $W$.
\end{Lemma}

\begin{proof} Given $\Om$ in a neighborhood of $M$ in $X$
as above, by Proposition~\ref{Proposition:extension} we may obtain a
1-form $\eta$ transversely meromorphic and closed
in $M$ such that $d(\Om|\_M) = \eta \wedge \Om|\_M$ and
$(\eta)_\infty = \La$ has order one. By transversality we may
extend $\eta$ to a small neighborhood $W$ of $M$ in $X$ as above
(see \cite{Camacho-Scardua}, \cite{Scardua2}).
\end{proof}

Under the hypothesis of Lemma\,\ref{Lemma:tubular} above if also
$X$ is a Stein manifold and $M = \po A$ is the smooth boundary of
a relatively compact open subset $A \subset X$ then by Levi's
Extension Theorem (\cite{Camacho-LinsNeto-Sad,Siu}) the 1-form $\eta$ extends to a meromorphic
(closed) 1-form $\tilde \eta$ on $W \cup \ov A$.

\noindent We proceed under the assumption
that  $M$ is
simply-connected.

\subsection{The irreducible case} As a first case assume that $ \tilde \La:=(\tilde\eta)_\infty$ is irreducible given by
an analytic equation say $\tilde\La = \{\tilde f = 0\}$ for some $\tilde
f\colon W \cup \ov A \to \bc$ holomorphic and reduced. Denote by
$\widetilde\fa$ the restriction of $\mathcal G$ to $W \cup \ov A$.

\begin{Claim}
\label{Claim:abelian}The holonomy group
$\Hol(\widetilde\fa,\widetilde L)$ of the leaf $\widetilde L=\tilde \Lambda \setminus \sing(\mathcal G)
\subseteq \Ga \cap (W \cup \ov A)$ of $\widetilde\fa$ is abelian
linearizable.
\end{Claim}

\begin{proof} We denote by $\widetilde\Om$ and
$\tilde\eta$ the restriction of $\Om$  to $W \cup \ov A$. We have
$(\tilde\eta)_\infty = (\tilde f=0)$ of order one so for a suitable
constant $a \in \bc^*$ we have $\tilde\eta - a\,\frac{d\tilde
f}{\tilde f} = \tilde\vr$ is holomorphic and closed in $W \cup \ov
A$. Since $M$ is simply-connected (by hypothesis) we may obtain a
holomorphic function $\tilde h$ in a neighborhood of $M$ in $W \cup
\ov A$ such that $\tilde\vr = d\tilde h$ in this neighborhood. By
Levi's Extension Theorem \cite{Siu} $\tilde h$ extends holomorphic
to $W \cup\ov A$ and we have $\tilde \vr = d\tilde h$ on $W \cup \ov
A$.

\noindent Thus $\tilde\eta - a\,\frac{d\tilde f}{\tilde f} = d\tilde
h$ and therefore $ \tilde\eta = a\, \frac{d\tilde f}{\tilde f} +
\frac{d(e^{\tilde h})}{e^{\tilde h}} = a\, \frac{d(\tilde f \,e
\frac{\tilde h}{a})}{(\tilde f\,e^{\frac{\tilde h}{a}})}\cdot $ We
may therefore assume that $\tilde \eta = a\,\frac{d\tilde f}{\tilde
f}$ in $W \cup \ov A$.

\vglue .1in

\noindent{\bf $1^{\text{st}}$ case}.\quad $2 \le a = k+1$ for some
$k \in \bn$. In this case $ d\widetilde\Om = \tilde\eta \wedge
\widetilde\Om \,\Rightarrow\, d\left(\frac{\widetilde\Om}{\tilde
f^{k+1}}\right) = 0.$ Thus $\widetilde\fa$ is given by a closed
meromorphic 1-form on $W \cup \ov A$. But, since
$\frac{\widetilde\Om}{\tilde f^{k+1}}$ has polar set $\{\tilde
f=0\}$ and of order $k+1 \ge 2$ we must integrate it (as above for
$\tilde\vr$) and obtain $\frac{\widetilde\Om}{\tilde f^{k+1}} =
d\widetilde F$ where $\widetilde F\colon W \cup \ov A \to \ov\bc$ is
meromorphic with polar set $(\widetilde f)_\infty = \{\tilde f=0\}$
of order $k$. Thus, actually we have $\widetilde F =
\frac{\widetilde G}{\tilde f^k}$ for some holomorphic function
$\widetilde G\colon W \cup\ov A \to \bc$. This gives
$$
\widetilde \Om = f^{k+1}\,d\left(\frac{\widetilde G}{\tilde
f^k}\right) = \frac{\tilde f^{k+1}(\tilde f^k d\widetilde G -
\widetilde G k\tilde f^{k-1}d\tilde f^k)}{\tilde f^{2k}}
\Rightarrow \widetilde \Om = \tilde f\,d\widetilde G - k\widetilde
G d\tilde f.
$$
Since $\widetilde\Om$ is the pull-back of a linear 2-dimensional
foliation the holonomy of its analytic leaves, and in particular
of the leaf $\{\tilde f=0\}$, is abelian linearizable.

\vglue .1in

\noindent{\bf $2^{\text{nd}}$ case}.\quad $a - 1 \notin \bn =
\{1,2,3,\dots\}$.

\noindent In this case by \cite{Scardua1} Section 3 the holonomy
group of the $\{\tilde f=0\}$ is abelian linearizable: we may find
local coordinates $(x_j,y_j) \in U_j$ covering a neighborhood of
$\{\tilde f=0\}$ in $W \cup\ov A$ such that: $\{\tilde f=0\} \cap
U_j : \{y_j=0\}$; \, $\widetilde\fa|_{U_j} : dy_j=0$; \,
$\widetilde\Om(x_j,y_j) = g_jdy_j$\,, \, $\tilde\eta(x_j,y_j) =
\frac{dg_j}{g_j} + a\,\frac{dy_j}{y_j}\,\cdot$ Since in each
intersection $U_i \cap U_j \ne\emptyset$ of two such charts we have
$$
\begin{cases}
g_i\,dy_i = g_j\,dy_j\\
a\,\frac{dy_i}{y_i} + \frac{dg_i}{g_i} = a\,\frac{dy_j}{y_j} +
\frac{dg_j}{g_j}
\end{cases}
$$
we must have (for $a-1 \notin \bn$) that $y_i = c_{ij}\,y_j$ for
some locally constant $c_{ij}$ in $U_i \cap U_j$
(\cite{Scardua1}).

The claim is proved.

\end{proof}
Now we notice that according to \cite{Scardua-Ito1}, since $A$
cannot contain a compact analytic subset of dimension $\ge1$, the
singular set $\sing(\widetilde\fa)$ is a finite set of points.
Moreover, since $ n \geq 3$ also according to \cite{Malgrange},
$\widetilde\fa$ has a local holomorphic first integral in a
neighborhood of each such singular point. Also,  if we denote by
$\widetilde\La = \Ga \cap (W \cup \ov A)$ then
$\pi_1(\widetilde\La)$ is naturally isomorphic to $\pi_1(\widetilde
L)$.  Using now Claim\,\ref{Claim:abelian} and the construction in
\cite{Camacho-LinsNeto-Sad} \S 2 (see for instance the proof of
Proposition~1 therein) we can construct a closed meromorphic
one-form $\tilde\xi$ defining $\tilde \fa$ in a neighborhood of
$\widetilde\La$ in $W \cup \ov A$ such that $\tilde\xi$ defines
$\widetilde\fa$ outside the polar set $(\tilde\xi)_\infty =
\widetilde\La$ which is of order one. Notice that, as before, given
a non-singular poin $p \in \tilde \Lambda$ and a transverse disc
$\Sigma$ to $\tilde \Lambda $ at $p=\tilde \Lambda \cap \Sigma$, we
may choose a holomorphic coordinate $z\in \Sigma$ such that the
holonomy group $\Hol(\tilde \fa, \tilde L, \Sigma, p)$ is linearized
by the coordinate $z\colon (\Sigma,p) \to (\mathbb C,0)$. Then, we
have $\tilde \xi\big|_{\Sigma} = \frac{dz}{z}$. Now, the dynamics of
$\widetilde\fa$ near $M$ allows to extend $\tilde\xi$ to a
neighborhood of $M$: indeed the holonomy of any leaf of $\fa$ not
contained in $\La$ is abelian linearizable (a subgroup of
$\Aff(\bc)$ with a fixed point at $0 \in \bc$) therefore we may
write locally $\tilde\xi = \frac{d\tilde y}{\tilde y}$ and extend
the function $y$ to a ``multivalued'' holomorphic function constant
on the leaves of $\fa$. Since the only non trivial dynamics of
$\fa$ is concentrated on $\La$ we conclude that this extension is
coherent.
 Thus we may take
$\tilde\xi$ as $\frac{d\tilde y}{\tilde y}$ for such an extension
which gives the desired extension of $\tilde\xi$ (see
\cite{Brunellaspheres} for a similar argumentation on the
extension).

Finally,
$\tilde\xi$ extends by transversality to a neighborhood od $M$ in
$W \cup \ov A$ and by Levi's Extension Theorem it extends to $W
\cup \ov A$.

Now we have obtained a closed meromorphic 1-form $\tilde\xi$ in $W
\cup \ov A$ with simple poles so that $(\tilde\xi)_\infty = \{\tilde
f=0\}$. Therefore, again, we may integrate $\tilde\xi -
\al\,\frac{d\tilde f}{\tilde f} = \tilde\vr$ for a suitable $\al \in
\bc-\{0\}$ and obtain $\tilde\xi = \al\,\frac{d\widetilde
F}{\widetilde F}$ for some holomorphic function $\widetilde F\colon
W \cup \ov A \to \bc$.

\subsection{General case} Now if $(\eta)_\infty = \La$ has more than one irreducible
component then the same argumentation above may show the existence
of $\tilde\xi$ with order one polar set $(\tilde\xi)_\infty =
\widetilde\La$ and which writes $\tilde\xi = \sum\limits_{j=1}^r
\al_j\,\frac{d\tilde f_j}{\tilde f_j}$ for holomorphic functions
$\tilde f_j\colon W \cup \ov A \to \bc$ and complex numbers $\al_j
\in \bc-\{0\}$. Indeed, there are two problems that we shall deal
with. One is the possibility that different components have
different types of holonomy groups (some abelian, some non-abelian).
The other is that we may find a coherent extension to the closed
one-forms that we construct in a neighborhood of each component. In
order to assure the coherency of the extension one only has to
observe that in a neighborhood of a component $\La_j \subset \La$
minus $\La_j$ itself there exists (up to multiplication by
constants) a unique closed holomorphic 1-form that defines
$\widetilde\fa$; this is because the quotient of two such closed
1-forms is a meromorphic first integral for $\widetilde\fa$. On the
other hand, since the holonomy group of $\La_j$ contains an
attractor, given any neighborhood $V_j$  of $\La_j$ in $M$, there
exist no meromorphic first integral for $\fa$ in
$V_j\backslash\La_j$ except constants.

 Thus we have nearly proved the following:

\begin{Theorem}
\label{Theorem:application} Let $\G$ be a codimension one
holomorphic foliation with singularities in a complex Stein
manifold $X^n$ of dimension $n \ge 3$. Let $A \subset X$ be a
relatively compact open subset with simply-connected smooth
boundary $M = \po A$ transverse to $\G$. Suppose that the induce
transversely holomorphic foliation $\fa = \G|\_M$ is
holomorphically transversely affine in $M\backslash \La$ for some
invariant subset $\La = \Ga \cap M$ where $\Ga \subset X$ is
analytic and invariant by $\G$. Also assume that each leaf of
$\fa$\,\, $L \subset \La$ contains an attractor on its holonomy
group.  Then $\G$ is given in a neighborhood of $\ov A$ in $X$ by
a closed meromorphic one form of logarithmic type.
\end{Theorem}

\begin{proof} We write $\La = \La_1 \cup \dots \cup \La_r$
in irreducible components $\La_j$ with $\La_i \cap \La_j =
\emptyset$, \, $\forall\, i\ne j$. Notice that the components are disjoint because $\fa$ is non-singular. If $r=1$ then we are done owing to
the above discussion. This is also the case if $r\ge2$ and the
$a_j = \Res_{\La_j}\,\eta$ satisfies either of the following
conditions: \vglue .1in

\noindent (1)\quad $2 \le a_j = k_j+1$, \quad $k_j \in \bn$, \quad
$\forall\, j \in \{1,\dots,r\}$. In this case $\G$ is given by a
closed rational 1-form $\frac{\widetilde\Om}{f_1^{k_1+1}\dots
f_r^{k_r+1}}$ and we proceed as above.

\vglue .1in

\noindent (2)\quad $a_j-1 \notin \bn$, \quad $\forall\,j \in
\{1,\dots,r\}$. In this case each $\La_j$ has abelian linearizable
holonomy and we may construct a simple poles closed meromorphic
1-form $\tilde\xi$ defining $\widetilde\fa$ in a neighborhood $W
\cup \ov A$ of $\ov A$ in $X$.

\noindent Therefore, it remains to consider the case $2 \le a_j =
k_j+1$, \quad $k_j \in \bn$ for $j \in \{1,\dots,j_0\}$ and $a_j-1
\notin \bn$, \quad $\forall\, j \in \{j_0+1,\dots,r\} \ne
\emptyset$.

\noindent Beginning with some component $\La_j$ with $j \in
\{j_0+1,\dots,r\}$ we construct a closed 1-form $\tilde\xi$ in a
neighborhood of $\La_j$ and extend it to a neighborhood of each
$\La_j$ with $j \in \{j_0+1,\dots,r\}$. It remains to extend
$\tilde\xi$ to a neighborhood of the $\La_j$ for which $j \in
\{1,\dots,j_0\}$. Suppose by contradiction that $\Hol(\fa,\La_1)$
is not abelian. In a neighborhood $V_1$ of $\La_1$ minus $\La_1$
there exists, up to multiplication by constant, a unique 1-form
$\eta_1$ such that $d\widetilde\Om = \eta_1 \wedge \widetilde\Om$:
given two such 1-form $\eta_1$ and $\eta_1'$ we have
$\eta_1-\eta_1' = h\widetilde\Om$ which is a closed 1-form in
$V_1\backslash\La_1$ defining $\fa$ if non trivial. Since the
holonomy of the leaf $\La_1$ of $\fa$ is non abelian such a closed
1-form cannot be identically zero and therefore $\eta = \eta_1$ in
$V_1\backslash\La_1$\,.

Thus, given $h_1$ in $V_1\backslash\La_1$ by $\widetilde\Om =
h_1\tilde\xi$ we have $d\widetilde\Om = \frac{dh_1}{h_1} \wedge
(h_1\tilde\xi) = \frac{dh_1}{h_1} \wedge \widetilde\Om$ and
therefore $\frac{dh_1}{h_1} = \tilde\eta$ in $V_1\backslash\La_!$\,.
This implies that, under the hypothesis that $\Hol(\fa,\La_1$ is not
abelian, we must have $\Res_{\La_1}\,\eta = 1$, contradiction.

Thus we have obtained, as in Claim\,\ref{Claim:abelian}, that
every component $\La_j \subset \La$ has abelian linearizable
holonomy group. As in the paragraph preceding the statement of
Theorem\,\ref{Theorem:application} above we may obtain a global
1-form $\tilde\xi$ in a neighborhood of $\ov A$ in $X$.

\end{proof}

Theorem~\ref{Theorem:main} is now a straightforward consequence of the above.

\begin{Remark}
\label{Remark:moregeneral}
{\rm We conclude from the proof above, that Theorem~\ref{Theorem:main} also holds if we replace the hypothesis
that $M$ is simply-connected by assuming that
 $H^1(M,\Om_\fa^\O) = 0$.
 It also holds for dimension $n=2$ if we assume that singularities
 of the foliation $\mathcal G$ in $\Gamma \cap A$ are the so called
 {\it first order singularities} (cf. \cite{Camacho-LinsNeto-Sad})
 meaning that they are non-dicritical generalized curves.
}
\end{Remark}

\bibliographystyle{amsalpha}

\begin{thebibliography}{31}
\frenchspacing



\bibitem{Brunella} M. Brunella; {\it On transversely
holomorphic flows I}; Inv. Math. 126, 265-279 (1996).


\bibitem{Brunellaspheres} M. Brunella; {\it On Holomorphic Foliations in
Complex Surfaces Transverse to a Sphere}; Bol. Soc. Bras. Mat., VoL 26, N. 2, 117-128
1995.



\bibitem{Loray2} {\rm M. Belliart, I. Liousse, F. Loray}:
 {\em Sur l'existence de points fixes attractifs pour les
 sous-groupes de Aut(${\Bbb C},0$)}, C. R. Acad. Sci. Paris S\'erie I Math. 324 (1997), no. 4, 443--446.


\bibitem{Camacho-Lins Neto} {\rm C. Camacho, A. Lins Neto}; Geometric theory of foliations; Birkhauser 1985.

\bibitem{Camacho-LinsNeto-Sad} C. Camacho, A. Lins Neto, P.
Sad; {\it Foliations with algebraic limit
sets}, Ann. of Math. (2) 136 (1992), 429--446.



\bibitem {Camacho-Scardua} {C. Camacho, B. Sc\'ardua};
Holomorphic foliations with Liouvillian first integrals, {\em
Ergodic Theory and Dynamical Systems} (2001), 21, pp.717-756.


\bibitem{Cerveau-Mattei} {\rm D. Cerveau, J.-F. Mattei}, {\em  Formes int\'egrables holomorphes
singuli\`eres}; Ast\'erisque, {\bf 97} (1982).

\bibitem{Cerveau-Moussu} D. Cerveau et R. Moussu: {\it Groupes d'automorphismes de
$(\Bbb C,0)$ et \'equations diff\'erentielles $ydy+\cdots =0$};
Bull. Soc. Math. France, {\bf 116} (1988), 459--488.

\bibitem{Ghys} E. Ghys; {\it On transversely holomorphic flows
II}; Inv. Math. 126, 281-286 (1996).


\bibitem{Godbillon} C. Godbillon; {\it Feuilletages:
\'Etudes Geom\'etriques}, \ Birkh\"auser, Berlin 1991.


\bibitem{Haefliger} A. Haefliger, D. Sundararaman; {\it
Complexifications of transversely holomorphic foliations}; Math.
Ann. 272, 23-27 (1985).

\bibitem{Loray} F. Loray, {\it Feuilletages holomorphes \`a holonomie r\'esoluble},
These Universit\'e de Rennes I, 1994.



\bibitem{Malgrange}  B. Malgrange: {\it Frobenius avec singularités, 1. Codimension un}, Public. Sc. I.H.E.S., 46 (1976), pp. 163-173.


\bibitem{Nakai} I. Nakai, {\it Separatrices for nonsolvable dynamics on $\bc,0$},
Ann. Inst. Fourier (Grenoble) 44 (1994), 569--599.



\bibitem{Scardua-Ito1} T. Ito, B. Sc\'ardua; {\it On holomorphic foliations transverse to spheres},
Moscow Mathematical Journal 5, no.2 (2005) 379-397.

\bibitem{Scardua-Ito2} T. Ito, B. Sc\'ardua; On the geometry of holomorphic flows and foliations having
transverse sections;  Geometriae Dedicata, Vol.111, No. 1, 23--42,
2005.



\bibitem{Scardua1} B. Sc\'ardua; \  {\it Transversely affine and transversely projective
holomorphic foliations}; Ann. Sc. \'Ecole Norm. Sup.,  4e. s\'erie
t.30, 1997, pp. 169-204.


\bibitem{Scardua2} B. Sc\'ardua;  {\it Integration of Complex Diff. Equations},  J. Dyn. Cont. Syst.
1, 1-50, (1999).


\bibitem{seke} Bobo Seke; {\it Sur les structures transversalement affines des
feuilletages de codimension un}; Ann. Inst. Fourier, Grenoble 30, 1
(1980), 1-29.


\bibitem{Siu} Y. Siu, Techniques of Extension of Analytic Object,
Lecture Notes in Pure and Appl. Math. 8,  Marcel Dekker, Inc., New
York, 1974.


\end{thebibliography}

\end{document}